\documentclass[11pt]{article}
\usepackage[T1]{fontenc}
\usepackage{lmodern,amsmath,amssymb,amsthm,mathtools,microtype}
\usepackage[a4paper,margin=27mm]{geometry}
\usepackage[hidelinks]{hyperref}
\usepackage{todonotes}
\newtheorem{theorem}{Theorem}[section]
\newtheorem{lemma}[theorem]{Lemma}
\newtheorem{proposition}[theorem]{Proposition}
\newtheorem{corollary}[theorem]{Corollary}
\theoremstyle{remark}
\theoremstyle{definition}
\newtheorem{definition}[theorem]{Definition}

\newcommand{\SL}{\mathsf{SL}}
\newcommand{\SAT}{\mathsf{SAT}}
\newcommand{\ThTwo}{\operatorname{Th}_2(\mathbb N)}
\newcommand{\AP}{\mathrm{AP}}
\newcommand{\Ag}{\mathrm{Ag}}
\newcommand{\Ac}{\mathrm{Ac}}
\newcommand{\St}{\mathrm{St}}
\newcommand{\Trk}{\mathrm{Trk}}
\newcommand{\Pth}{\mathrm{Pth}}
\newcommand{\Str}{\mathrm{Str}}
\newcommand{\Var}{\mathrm{Var}}
\newcommand{\ap}{\mathsf{ap}}
\newcommand{\tr}{\mathsf{tr}}
\newcommand{\Mem}{\mathsf{Mem}}
\newcommand{\Eq}{\mathsf{Eq}}
\newcommand{\Succ}{\mathsf{Succ}}
\newcommand{\Add}{\mathsf{Add}}
\newcommand{\Mult}{\mathsf{Mult}}

\newcommand{\Ext}{\mathsf{Ext}}
\newcommand{\NN}{\mathbb N}
\newcommand{\Pow}{\mathcal P}
\newcommand{\CGS}{\mathcal G}
\newcommand{\EEx}[1]{\langle\!\langle #1\rangle\!\rangle}
\newcommand{\AAll}[1]{[\![#1]\!]}
\newcommand{\Next}{\mathbf X}

\newcommand{\SLBG}{\mathsf{SL[BG]}}
\newcommand{\SLBGX}{\mathsf{SL_{\Next}[BG]}}
\newcommand{\B}{\mathsf B}
\newcommand{\Good}{\mathsf{Good}}
\newcommand{\Test}{\mathsf{Test}}

\title{Exact Complexity of the Satisfiability Problem for Strategy Logic}
\author{Tikhon Pshenitsyn}
\date{\today}

\begin{document}
\maketitle
\begin{abstract}
	
We show that the satisfiability problem for Strategy Logic introduced by Mogavero, Murano, and Vardi is $\Pi^1_\infty$-complete, and, more strongly, computably isomorphic to true second-order arithmetic. The lower bound is established for the next-time Boolean-goal fragment of Strategy Logic. Consequently, Strategy Logic is not recursively axiomatizable, even with effectively defined $\omega$-rules.

\end{abstract}

\section{Introduction}

Strategy Logic is an extension of \(\mathsf{LTL}\) that introduces strategies as first-class objects that can be explicitly quantified over and assigned to agents. It was originally introduced for two-player turn-based games by Chatterjee, Henzinger, and Piterman~\cite{ChatterjeeHenzingerPiterman2010} and subsequently developed for multi-agent concurrent game structures by Mogavero, Murano, and Vardi~\cite{MogaveroMuranoVardi2010}. This paper is concerned with the latter logic, denoted by $\SL$. 

$\SL$ subsumes well-known logics for reasoning about branching time and strategies, namely, $\mathsf{CTL}$, $\mathsf{CTL}^\ast$, $\mathsf{ATL}$, $\mathsf{ATL}^\ast$. Moreover, $\SL$ allows one to formalize game-theoretic concepts that the latter logics fail to express. However, the expressivity of $\SL$ comes at a high cost. $\SL$ has a decidable but non-elementary (TOWER-complete) model-checking problem~\cite{MogaveroMuranoPerelliVardi2014}.  Its satisfiability problem was studied by Mogavero, Murano, Perelli, and Vardi~\cite{MogaveroMuranoPerelliVardi2017}, who claimed its high undecidability, namely, \(\Sigma^1_1\)-hardness. For the proof, they used the recurrent tiling problem. On the positive side, they isolated the one-goal fragment \(\mathsf{SL[1G]}\), and showed its satisfiability problem to be decidable and \(2\mathrm{EXPTIME}\)-complete~\cite{MogaveroMuranoPerelliVardi2012,MogaveroMuranoPerelliVardi2017}. Further decidable fragments have since been investigated, including the PSPACE-complete flat conjunctive-goal fragment, in which strategy quantification does not occur under temporal operators~\cite{AcarBenerecettiMogavero2019} and non-recurrent fragments of \(\mathsf{SL[1G]}\)~\cite{BenerecettiMogaveroPeron2022}.  Satisfiability also becomes decidable when the number of available actions is bounded in advance~\cite{LaroussinieMarkey2013}.

The complexity of $\SL$ satisfiability was recently revisited by Catta, Galimullin, and Murano~\cite{CattaGalimullinMurano2025}, who observed that the \(\Sigma^1_1\)-hardness proof from \cite{MogaveroMuranoPerelliVardi2017} contains a gap. The surviving part of the tiling argument establishes $\Pi^0_1$-hardness of the $\SL$ satisfiability problem \cite{MogaveroMuranoPerelliVardi2017,CattaGalimullinMurano2025}. The paper~\cite{CattaGalimullinMurano2025} thus poses the question of whether $\SL$ is recursively axiomatizable: $\Pi^0_1$-hardness of satisfiability is consistent with that possibility, unlike $\Sigma^1_1$-hardness.

This paper contributes to this discussion by showing that satisfiability in \(\SL\) is $\Pi^1_\infty$-complete. More precisely, the main result is stated below.
 
 \begin{theorem}\label{theorem:main}
 	The set of satisfiable \(\SL\) sentences is computably isomorphic to true second-order arithmetic. The same holds for the next-time Boolean-goal fragment of $\SL$.
 \end{theorem}
 This complexity characterization entails that \(\SL\) does not admit a sound and complete effective axiomatization, even with $\omega$-rules. (We assume that $\omega$-rules are effective in the sense that their premises are defined in a computable way.) 
 
 Our lower bound proof is an interpretation of second-order arithmetic in $\SL$. It shows that strategies can be used as genuine second-order objects, as their semantics suggests. The lower bound is attained already for the next-time Boolean-goal fragment of $\SL$, i.e.~the interpretation uses only the next-time operator $\Next$ among temporal operators and obeys the Boolean-goal restriction defined below. This shows that a rather small part of $\SL$ syntax beyond its one-goal fragment suffices to restore the full second-order power of strategy quantifiers.

\section{Preliminaries}\label{sec:prelim}

We recall the syntax and semantics of $\SL$, following~\cite{MogaveroMuranoPerelliVardi2017}.

\begin{definition}
	A \emph{concurrent game structure (CGS)} is a tuple
	\(\CGS=\langle \AP,\Ag,\Ac,\St,\tr,\ap,s_0\rangle\)
	where \(\AP\) (atomic propositions) and \(\Ag\) (agents) are finite nonempty sets, \(\Ac\) (actions) and \(\St\) (states) are at most countable nonempty sets, \(s_0\in\St\) is the initial state, \(\tr:\St\times\Ac^{\Ag}\to\St\) is the transition function, and \(\ap:\St\to 2^{\AP}\) is the labeling function.
\end{definition}

\begin{definition}
	An element \(\delta\in\Ac^{\Ag}\) is called a \emph{decision}.  A \emph{track} (\emph{path}) in $\CGS$ is a non-empty finite (infinite resp.) sequence
	\(
	\rho=t_0\ldots t_n
	\)
	(\(\rho = t_0 t_1 \ldots \) resp.)
	of states such that, for every $i<n$ ($i \in \NN$ resp.), there exists a decision
	\(
	\delta_i\in\Ac^{\Ag}
	\)
	with
	\(
	\tr(t_i,\delta_i)=t_{i+1}.
	\)
	We write $\Trk_{\CGS}$ ($\Pth_{\CGS}$) for the set of all tracks (paths resp.) in $\CGS$. Moreover, for $s \in \St$, $\Trk_{\CGS}(s)$ ($\Pth_{\CGS}(s)$) denotes the set of all tracks (paths resp.) starting with $s$.
	
	If $\rho = t_0 t_1 \ldots$ is a path, then $\rho_{\le n} = t_0 \ldots t_n$. 
\end{definition}

\begin{definition}
	If $\rho = t_0 \ldots t_m$ and $\tau = t_m \ldots t_n$ are two tracks, then $\rho \star \tau = t_0 \ldots t_n$. 
\end{definition}

\begin{definition}
	A \emph{strategy} is a partial function
	\(
	f:\Trk_{\CGS}\rightharpoonup \Ac
	\). \(\Str_{\CGS}\) denotes the set of all strategies. $f$ is $s$-total if $\Trk_{\CGS}(s) \subseteq \operatorname{dom}(f)$; the set of $s$-total strategies is denoted by $\Str_{\CGS}(s)$.
\end{definition}

\begin{definition}
	If \(f\in\Str_{\CGS}\) and \(\rho\in\Trk_{\CGS}\), define the \emph{shift} of \(f\) along \(\rho\) by
	\[
	f_{\rightarrow\rho}(\tau):=f(\rho \star \tau)
	\quad(\tau\in\Trk_{\CGS}).
	\]
\end{definition}
Clearly, $f_{\rightarrow\rho}(\tau)$ is defined iff the first state of $\tau$ coincides with the last state of $\rho$ and $\rho \star \tau \in \operatorname{dom}(f)$. In particular, if $\rho = \rho_0 \ldots \rho_m$ and $f$ is $\rho_0$-total, then $f_{\rightarrow\rho}$ is $\rho_m$-total.

Let $\Var$ be a countable set of strategy variables, and let $\AP$ and $\Ag$ be sets of propositions and agents disjoint from $\Var$. $\SL$ formulas are built according to the following grammar, where $x \in \Var$, $p \in \AP$, $a \in \Ag$:
\[
\begin{aligned}
\varphi ::={}&p\mid\neg\varphi\mid\varphi\wedge\varphi\mid\varphi\vee\varphi
 \mid\Next\varphi\mid\varphi\mathsf U\varphi\mid\varphi\mathsf R\varphi\\
 &\mid\EEx{x}\varphi\mid\AAll{x}\varphi\mid(a,x)\varphi.
\end{aligned}
\]

\begin{definition}
	The set $\operatorname{free}(\varphi)$ of free variables and agents of a formula $\varphi$ is defined inductively:
	\begin{align*}
		\operatorname{free}(p)&=\varnothing,
		&& p \in \AP, \\
		\operatorname{free}(\neg\varphi)&=\operatorname{free}(\varphi),\\
		\operatorname{free}(\varphi\circ\psi)&=\operatorname{free}(\varphi)\cup\operatorname{free}(\psi),
		&&\circ\in\{\land,\lor\},\\
		\operatorname{free}(\Next\varphi)&=\Ag\cup\operatorname{free}(\varphi),\\
		\operatorname{free}(\varphi\circ\psi)&=\Ag\cup\operatorname{free}(\varphi)\cup\operatorname{free}(\psi),
		&&\circ\in\{\mathsf U,\mathsf R\},\\
		\operatorname{free}(Qx\,\varphi)&=\operatorname{free}(\varphi)\setminus\{x\},
		&&Qx\in\{\EEx{x},\AAll{x}\},
		\\
		\operatorname{free}((a,x)\varphi)&=
		\begin{cases}
			\operatorname{free}(\varphi),&a\notin\operatorname{free}(\varphi),\\
			(\operatorname{free}(\varphi)\setminus\{a\})\cup\{x\},&a\in\operatorname{free}(\varphi).
		\end{cases}
	\end{align*}
\end{definition}

A formula $\varphi$ is a \emph{sentence} if $\operatorname{free}(\varphi)=\varnothing$; it is \emph{agent-closed} if $\operatorname{free}(\varphi)\cap\Ag=\varnothing$.

\begin{definition}
	Given a CGS \(\CGS=\langle \AP,\Ag,\Ac,\St,\tr,\ap,s_0\rangle\), an \emph{assignment} is a partial map
	\(
	\chi:\Var\cup\Ag\rightharpoonup\Str_{\CGS}.
	\)
	For \(\ell\in\Var\cup\Ag\) and \(f\in\Str_{\CGS}\), the notation \(\chi[\ell\mapsto f]\) denotes the assignment with domain \(\operatorname{dom}(\chi)\cup\{\ell\}\) that maps \(\ell\) to \(f\) and agrees with \(\chi\) everywhere else. In particular, \(\chi[a\mapsto\chi(x)]\) assigns to agent \(a\) the strategy currently stored in \(x\). The assignment is \emph{complete} if every agent is in its domain: \(\operatorname{dom}(\chi) \supseteq \Ag\). The assignment is $s$-total if $\chi(\ell)$ is $s$-total for every $\ell \in \operatorname{dom}(\chi)$.
\end{definition}

\begin{definition}
	The \emph{shift} of an assignment $\chi$ along a track $\rho$ is the assignment $\chi_{\rightarrow\rho}$ with the same domain as $\chi$ such that $\chi_{\rightarrow\rho}(\ell) = (\chi(\ell))_{\rightarrow\rho}$ for each $\ell \in \operatorname{dom}(\chi)$. 
\end{definition}

\begin{definition}
	Let \(s \in \St\) be a state and \(\chi\) a complete $s$-total assignment. A \emph{\((\chi,s)\)-play} is the path
	\(\operatorname{play}(\chi,s) = \pi = \pi_0 \pi_1 \pi_2\ldots\in\Pth_{\CGS}\) where \(\pi_0=s\) and \(\pi_{i+1}=\tr(\pi_i,\delta_i)\) for \(\delta_i(a)=\chi(a)(\pi_{\le i})\). 
	Define the \emph{global translation} of \((\chi,s)\) to position \(i \in \NN \) by \((\chi,s)^i=(\chi_{\rightarrow\pi_{\le i}},\pi_i)\).
\end{definition}

Let \(\varphi\) be an \(\SL\)-formula, \(s\in\St\), and let \(\chi\) be an $s$-total assignment such that \(\operatorname{free}(\varphi) \subseteq \operatorname{dom}(\chi)\). The satisfaction relation \(\CGS,\chi,s\models\varphi\) is defined inductively.  The Boolean clauses are standard; the remaining clauses are
\begin{align*}
	\CGS,\chi,s\models p
	&\iff p\in\ap(s),\\
	\CGS,\chi,s\models\EEx{x}\psi
	&\iff \text{there exists }f\in\Str_{\CGS}(s)
	\text{ such that }\CGS,\chi[x\mapsto f],s\models\psi,\\
	\CGS,\chi,s\models\AAll{x}\psi
	&\iff \text{for every }f\in\Str_{\CGS}(s),
	\ \CGS,\chi[x\mapsto f],s\models\psi,\\
	\CGS,\chi,s\models(a,x)\psi
	&\iff \CGS,\chi[a\mapsto\chi(x)],s\models\psi.
\end{align*}
If $a\notin\operatorname{free}(\psi)$, the binding is immaterial and its clause is read simply as $\CGS,\chi,s\models\psi$; in particular, it does not require $x\in\operatorname{dom}(\chi)$. 

The semantics of temporal operators is defined under the condition that \(\chi\) is complete.
\begin{align*}
	\CGS,\chi,s\models\Next\psi
	&\iff \CGS, (\chi,s)^1 \models\psi,\\
	\CGS,\chi,s\models\psi_1\,\mathsf U\,\psi_2
	&\iff \exists i\in\NN \ \bigl(
	\CGS, (\chi,s)^i \models\psi_2
	\ \land\
	\forall j<i\ \CGS,(\chi,s)^j \models\psi_1
	\bigr),\\
	\CGS,\chi,s\models\psi_1\,\mathsf R\,\psi_2
	&\iff \forall i\in\NN \ \bigl(
	\CGS,(\chi,s)^i \models\psi_2
	\ \lor\
	\exists j<i\ \CGS,(\chi,s)^j \models\psi_1
	\bigr).
\end{align*}

For a sentence $\varphi$ we abbreviate $\CGS,\varnothing,s\models\varphi$ by $\CGS,s\models\varphi$. We say that $\CGS$ satisfies a sentence $\varphi$ ($\CGS \models \varphi$ in symbols) if $\CGS, s_0 \models \varphi$ ($s_0$ is the initial state of $\CGS$). A sentence $\varphi$ is satisfiable if there is a CGS $\CGS$ that satisfies $\varphi$.

\paragraph{The Boolean-goal fragment.}
A \emph{quantification prefix} over a finite set $V\subseteq\Var$ is a word $\wp$ of strategy quantifiers containing exactly one quantifier, either $\EEx{x}$ or $\AAll{x}$, for each $x\in V$. A \emph{binding prefix} is a word $\flat$ of bindings $(a,x)$ in which each agent $a\in\Ag$ occurs exactly once. Different agents may be bound to the same variable. 

The \emph{Boolean-goal fragment}, denoted by $\SLBG$, is defined by the grammar
\begin{align*}
\varphi &::=p\mid\neg\varphi\mid\varphi\wedge\varphi\mid\varphi\vee\varphi
 \mid\Next\varphi\mid\varphi\mathsf U\varphi\mid\varphi\mathsf R\varphi\mid\wp\beta,\\
\beta &::=\flat\varphi\mid\neg\beta\mid\beta\wedge\beta\mid\beta\vee\beta,
\end{align*}
where $p\in\AP$, $\flat$ is a binding prefix, and in $\wp\beta$ the prefix $\wp$ quantifies exactly $\operatorname{free}(\beta)$. The \emph{next-time Boolean-goal fragment} $\SLBGX$ additionally excludes $\mathsf U$ and $\mathsf R$ from this grammar.

\paragraph{Computability and second-order theories.}

For sets $E,F\subseteq\NN$, a \emph{computable many-one reduction} from $E$ to $F$ is a total computable function $f:\NN\to\NN$ such that $n\in E$ iff $f(n)\in F$. We write $E\le_m F$; if $f$ is injective, we write $E\le_1 F$ and say that $E$ is \emph{one-one reducible} to $F$. A \emph{computable isomorphism} is such a reduction that is a bijection of $\NN$; its inverse is then computable as well. Myhill's isomorphism theorem states that two sets of natural numbers are computably isomorphic if and only if each is one-one reducible to the other; see~\cite{Rogers1967}. 

We use the relational language $\mathsf{RelAr}$ for second-order arithmetic, with a constant $0$, relations $S(\cdot,\cdot),A(\cdot,\cdot,\cdot),M(\cdot,\cdot,\cdot)$, equality of first-order terms, and membership $\in$. First-order variables range over elements and second-order variables range over subsets of the first-order domain.  The standard second-order model of arithmetic is denoted by $\mathcal N$. For a second-order formula $\varphi$, denote its free first-order (second-order) variables by $\operatorname{free}_1(\varphi)$ ($\operatorname{free}_2(\varphi)$ resp.) The set of closed second-order sentences true in $\mathcal N$ is denoted by $\ThTwo$. The notion of $\Pi^1_\infty$-completeness refers to the many-one degree of $\ThTwo$. However, for $\SL$ satisfiability, we actually prove the stronger result of computable isomorphism between it and $\ThTwo$.

For $\mathcal D=(D,0^D,S^D,A^D,M^D)$ a $\mathsf{RelAr}$-structure with nonempty domain $D$, a second-order \emph{assignment} in $\mathcal D$ is a pair $(\nu_1,\nu_2)$ of partial maps from first-order variables to $D$ and from second-order variables to $\Pow(D)$, respectively. To evaluate a formula $\varphi$, their domains must contain $\operatorname{free}_1(\varphi)$ and $\operatorname{free}_2(\varphi)$. The relation $\mathcal D,\nu_1,\nu_2\models\varphi$ ($\varphi$ is true in $\mathcal D$ under the assignment $(\nu_1,\nu_2)$) is defined as usual.

We denote the set of satisfiable sentences in a logic $\mathcal L$ by $\SAT(\mathcal L)$.

\section{Reduction from Second-Order Arithmetic}\label{sec:lower}

The upper bound from Theorem \ref{theorem:main} is rather trivial. For the sake of completeness, let us briefly sketch the proof.

\begin{proposition}\label{prop:upper}
	\(\SAT(\SL)\le_1\ThTwo\).
\end{proposition}
\begin{proof}[Proof sketch]
	
	Every CGS has at most countable sets of states and actions, so its states, actions, transition function, labeling function, and initial state can be coded by subsets of \(\NN\). (Note that the transition function $\tr: \St \times \Ac^{\Ag} \to \St$ maps a countable set to a countable set.) Tracks can be coded by natural numbers, therefore, a strategy \(f:\Trk\to\Ac\) is a second-order object. Once the strategies of all agents are fixed, every finite prefix of the induced play is arithmetically definable from the model and strategies.  \(\Next\) and \(\mathsf{U}\) require only first-order quantification over positions of the play. Finally, the satisfiability statement is of the form ``there exists a CGS such that\dots'', so its coding requires an existential second-order quantifier. This allows one to transform a given $\SL$ formula $\vartheta$ into a second-order arithmetic sentence $\vartheta^\circledast$ whose truth in $\mathcal N$ is equivalent to satisfiability of $\vartheta$. The mapping $\vartheta \mapsto \vartheta^\circledast$ is clearly injective (each atomic formula and connective are translated homomorphically.) 
\end{proof}

Thus the lower bound is of the main interest:

\begin{theorem}\label{thm:injective-lower}
	$\ThTwo \le_1 \SAT(\SLBGX)$.
\end{theorem}

The idea of the proof is based on interpreting second-order arithmetic in \(\SLBGX\). An $s$-total strategy can be viewed as a labeling of the countably branching tree of finite histories from $s$, and thus as a genuine second-order object. We encode natural numbers by states obtained after one move in a game and subsets of natural numbers by states obtained after two moves in a game. This is why using only the next-time operator suffices for our construction.

Let us start describing the interpretation. Fix the agent set and the proposition set
\[
\Ag=\{a_1,a_2,a_3,b,c\}
\quad\text{and}\quad
\AP=\{p_b,p_c,p_S,p_A,p_M\}.
\]
We order the agents as $a_1,a_2,a_3,b,c$ and identify each decision with its tuple of actions in this order. We abbreviate $\tr(s,(\alpha_1,\alpha_2,\alpha_3,\alpha_b,\alpha_c))$ as $\tr(s,\alpha_1,\alpha_2,\alpha_3,\alpha_b,\alpha_c)$. Similarly, we identify an assignment $\chi$ where $\operatorname{dom}(\chi) = \Ag$ with the tuple of strategies $(\chi(a_1),\chi(a_2),\chi(a_3),\chi(b),\chi(c))$.

\begin{definition}
	For a 5-tuple of strategy variables $\bar x=(x_1,x_2,x_3,x_b,x_c)$, put
	\[
	\B(\bar x)=(a_1,x_1)(a_2,x_2)(a_3,x_3)(b,x_b)(c,x_c).
	\]
	For $i\in\Ag$, the tuple $\bar x[i\leftarrow v]$ replaces the coordinate for agent $i$ by $v$. The quantifier block $\AAll{\bar x}$ ($\EEx{\bar x}$) universally (existentially resp.) quantifies the five variables.
\end{definition}

\begin{definition}
	For each $i\in\{b,c\}$ define the closed sentences
	\begin{align}
		I_i &:= \AAll{\bar x}\AAll{\bar y}
		\bigl(\B(\bar x)\Next p_i
		\leftrightarrow
		\B(\bar y[i\leftarrow x_i])\Next p_i\bigr),
		\label{eq-local-ind}
		\\
		P_i &:= (\EEx{\bar x}\B(\bar x)\Next p_i)
		\wedge(\EEx{\bar y}\B(\bar y)\Next\neg p_i).
		\label{eq-local-onto}
	\end{align}
	Let \(\Test:=I_b\wedge P_b\wedge I_c\wedge P_c\).
\end{definition}

Informally, $\Test$ is true in a state $s$ iff the truth of $p_b$ (respectively, $p_c$) after one move in a game starting at $s$ depends only on $b$'s (respectively, $c$'s) action, and each of $p_b$ and $p_c$ can be made either true or false by some move. Formally:

\begin{lemma}\label{lem-local}
A state $s\in\St$ satisfies $\Test$ if and only if there are surjective functions
\(\beta_b^s,\beta_c^s:\Ac\to\{0,1\}\) such that, for every decision $\bar\alpha\in \Ac^5$ and $i\in\{b,c\}$,
\begin{equation}\label{eq-local-bit}
p_i\in\ap(\tr(s,\bar\alpha))
\quad\Longleftrightarrow\quad
\beta_i^s(\alpha_i)=1.
\end{equation}
\end{lemma}

\begin{definition}
	Define three closed sentences:
	\begin{align}
		L&:=\AAll{\bar x}\B(\bar x)\Next\Test,\label{eq-all-local}\\
		J_b&:=\AAll{\bar x}\AAll{v}
		\bigl(\B(\bar x)\Next\Next p_c
		\leftrightarrow\B(\bar x[b\leftarrow v])\Next\Next p_c\bigr),\label{eq-Jb}\\
		J_c&:=\AAll{\bar x}\AAll{v}
		\bigl(\B(\bar x)\Next\Next p_b
		\leftrightarrow\B(\bar x[c\leftarrow v])\Next\Next p_b\bigr).\label{eq-Jc}
	\end{align}
	The variable $v$ is fresh in each formula. Finally set \(\Good:=L\wedge J_b\wedge J_c\).
\end{definition}
Overall, the purpose of $\Good$ is to enforce enough ``richness'' in a CGS $\CGS$ to interpret natural numbers and their subsets in it.

\begin{lemma}\label{lem-independence}
Suppose $\CGS,s_0\models\Good$. There is a function $q:\Ac^3\to\St$ such that
\\
\(
\tr(s_0,\alpha_1,\alpha_2,\alpha_3,\alpha_b,\alpha_c)
=q(\alpha_1,\alpha_2,\alpha_3)
\)
for all decisions. Besides, every state in the image of $q$ satisfies $\Test$.
\end{lemma}
\begin{proof}

Fix actions $\alpha_1,\alpha_2,\alpha_3,\alpha_c$ and two actions $u,v$ for $b$. Let
\[
s=\tr(s_0,\alpha_1,\alpha_2,\alpha_3,u,\alpha_c),\qquad
t=\tr(s_0,\alpha_1,\alpha_2,\alpha_3,v,\alpha_c).
\]
Suppose $s\ne t$. Since both $s$ and $t$ satisfy $\Test$, choose actions $\gamma_s,\gamma_t$ with
$\beta_c^s(\gamma_s)=0$ and $\beta_c^t(\gamma_t)=1$ (Lemma \ref{lem-local}). Define a single strategy $f_c$ by
\(f_c(s_0)=\alpha_c\), \(f_c(s_0s)=\gamma_s\), \(f_c(s_0t)=\gamma_t\). It can be extended arbitrarily to every other track. Choose some strategies $f_1,f_2,f_3$ such that $f_i(s_0) = \alpha_i$, and two strategies $f_b^1,f_b^2$ for $b$ such that $f_b^1(s_0) = u$, $f_b^2(s_0) = v$. 

If the strategies $(f_1,f_2,f_3,f_b^1,f_c)$ are played, the transition function takes $s_0$ to $s$ on the first move. Since $\CGS, s \models \Test$, it holds that $\CGS,(f_1,f_2,f_3,f_b^1,f_c)_{\rightarrow s_0s},s \models \Next p_c$ iff $\beta^s_c(\gamma_s) = 1$---which is false. Hence, $\CGS,(f_1,f_2,f_3,f_b^1,f_c),s_0 \not\models \Next \Next p_c$.
Similarly, one can show that $\CGS,(f_1,f_2,f_3,f_b^2,f_c),s_0 \models \Next \Next p_c$, because $\beta^t_c(\gamma_t) = 1$.

Therefore, $\CGS, s_0 \not \models J_b$, i.e., the formula after the quantifier block is false when one assigns the strategies $(f_1,f_2,f_3,f_b^1,f_c)$ to $\bar x$ and $f_b^2$ to $v$. This is a contradiction. It gives us that $s=t$, so indeed the value $\tr(s_0,\alpha_1,\alpha_2,\alpha_3,\alpha_b,\alpha_c)$ does not depend on $\alpha_b$. Symmetrically, one can show that it does not depend on $\alpha_c$, hence the desired function $q$ exists. Finally, the second statement of the lemma directly follows from the definition of $L$.
\end{proof}

\paragraph{Membership and equality.}

Hereinafter, we fix $\CGS = \langle \AP,\Ag,\Ac,\St,\tr,\ap,s_0\rangle$ such that $\CGS,s_0\models\Good$, so the function $q$ of Lemma~\ref{lem-independence} is available. Reserve a fresh strategy variable $e$, which will be existentially quantified in the final formula. We also fix an $s_0$-total assignment $\varepsilon$ mapping $e$ to a strategy. For an agent-closed formula $\psi(z_1,\ldots,z_k)$ whose free variables lie in $\{e,z_1,\ldots,z_k\}$, with $z_j\ne e$, and strategies $f_1,\ldots,f_k\in\Str_{\CGS}(s_0)$, we simply write \(\psi(f_1,\ldots,f_k)\) as a shorthand for
\[
\CGS,\varepsilon[z_1\mapsto f_1,\ldots,z_k\mapsto f_k],s_0
\models\psi(z_1,\ldots,z_k).
\]

\begin{definition}\label{definition:d_f}
	For $f\in\Str_{\CGS}(s_0)$, let
	\(
	d_f:=q(f(s_0),\varepsilon(e)(s_0),\varepsilon(e)(s_0))\). Let 
	\[D:=\{d_f \mid f \in \Str_{\CGS}(s_0)\}.\]
\end{definition}
The set $D$ is nonempty and at most countable.

\begin{definition}
	Let
	\begin{align}
		\Mem(Y,x)&:=\B(x,e,e,Y,e)\Next\Next p_b,\label{eq-mem}\\
		\Eq(x,y)&:=\AAll{Y}\bigl(\Mem(Y,x)\leftrightarrow\Mem(Y,y)\bigr).
		\label{eq-eq}
	\end{align}
\end{definition}

It is straightforward to check from the definitions above that, if $\CGS,s_0\models\Good$, then for all $F,f,g\in\Str_{\CGS}(s_0)$,
\begin{equation}\label{eq-mem-sem}
	\Mem(F,f)
	\quad\Longleftrightarrow\quad
	\beta_b^{d_f}(F(s_0d_f))=1.
\end{equation}

\begin{definition}
	For $F$ an $s_0$-total strategy, let \(\hat F:=\{d_f \in D \mid f \in \Str_{\CGS}(s_0) ~ \text{and} ~ \Mem(F,f)\}\).
\end{definition}

\begin{lemma}\label{lem-powerset}
If $\CGS,s_0\models\Good$, then $\Pow(D) = \{\hat F \mid F \in \Str_{\CGS}(s_0)\}$.
\end{lemma}
\begin{proof}
Choose for each $d\in D$ actions $\gamma_d^0,\gamma_d^1$ with $\beta_b^d(\gamma_d^j)=j$. For any $U\subseteq D$, define a strategy $F_U$ as follows:
\[
F_U(s_0d)=\begin{cases}\gamma_d^1,&d\in U,\\\gamma_d^0,&d\notin U.\end{cases}
\]
For all other tracks, define $F_U$ arbitrarily. Then \(d \in \widehat{F_U}\) iff \(\beta_b^d(F_U(s_0d))=1\) iff \(d \in U\). Hence $U = \widehat{F_U}$, as required.
\end{proof}

\begin{corollary}\label{corollary-eq}
	If $\CGS,s_0\models\Good$, then $\Eq(f,g) \Longleftrightarrow d_f=d_g$.
\end{corollary}

\begin{proof}
	\leavevmode
	
	$(\Longrightarrow)$ If $d_f\ne d_g$, let $U = \{d_f\}$. Then $\Mem(F_U,f)$ and $\neg\Mem(F_U,g)$ hold, therefore, \(\neg\Eq(f,g)\).
	
	$(\Longleftarrow)$ Follows from \eqref{eq-mem-sem}.
\end{proof}

\paragraph{Interpreting arithmetic.}

Now, we interpret arithmetical relations in $\SLBGX$. This part is fairly standard---we essentially axiomatize the standard model of arithmetic using second-order formulae.

\begin{definition}
	Let
	\begin{align}
		\Succ(x,y)&:=\B(x,y,e,e,e)\Next p_S,\label{eq-succ}\\
		\Add(x,y,z)&:=\B(x,y,z,e,e)\Next p_A,\label{eq-add}\\
		\Mult(x,y,z)&:=\B(x,y,z,e,e)\Next p_M.\label{eq-mult}
	\end{align}
\end{definition}

\begin{definition}
	For $R\in\{\Succ,\Add,\Mult\}$ of arity $k\in\{2,3\}$ let
	\begin{equation*}
		\Ext_R:=\AAll{x_1}\cdots\AAll{x_k}\AAll{y_1}\cdots\AAll{y_k}
		\left(\bigwedge_{j=1}^k\Eq(x_j,y_j)
		\rightarrow (R(x_1,\ldots,x_k)\leftrightarrow R(y_1,\ldots,y_k))\right).
	\end{equation*}
	Let $\Ext=\Ext_{\Succ}\wedge\Ext_{\Add}\wedge\Ext_{\Mult}$. Let $\mathsf{PA}_2$ be the conjunction of the following statements:
	\begin{align}
		&\forall x\exists y\ S(x,y),\label{eq-s-total}\\
		&\forall x\forall y\forall z\ (S(x,y)\wedge S(x,z)\to y=z),\label{eq-s-fun}\\
		&\forall x\forall y\forall z\ (S(x,z)\wedge S(y,z)\to x=y),\label{eq-s-inj}\\
		&\forall x\ \neg S(x,0),\label{eq-s-zero}\\
		&\forall U\left[\left(0\in U\wedge
		\forall x\forall y\,(x\in U\wedge S(x,y)\to y\in U)\right)
		\to\forall x\ x\in U\right],\label{eq-induction}\\
		&\forall x\forall y\exists z\ A(x,y,z),\label{eq-a-total}\\
		&\forall x\forall y\forall z\forall w\ (A(x,y,z)\wedge A(x,y,w)\to z=w),\label{eq-a-fun}\\
		&\forall x\ A(x,0,x),\label{eq-a-zero}\\
		&\forall x\forall y\forall y'\forall z\forall z'\
		(S(y,y')\wedge A(x,y,z)\wedge S(z,z')\to A(x,y',z')),\label{eq-a-rec}\\
		&\forall x\forall y\exists z\ M(x,y,z),\label{eq-m-total}\\
		&\forall x\forall y\forall z\forall w\ (M(x,y,z)\wedge M(x,y,w)\to z=w),\label{eq-m-fun}\\
		&\forall x\ M(x,0,0),\label{eq-m-zero}\\
		&\forall x\forall y\forall y'\forall z\forall w\
		(S(y,y')\wedge M(x,y,z)\wedge A(z,x,w)\to M(x,y',w)).\label{eq-m-rec}
	\end{align}
\end{definition}

\begin{definition}
	We define the translation $\dagger$ on the formulas of the two-sorted relational language of second-order arithmetic. Choose disjoint infinite families of strategy variables to represent the number variables and the set variables, neither containing $e$. Let \(x^\dagger=x\) for $x$ a first-order variable, \(0^\dagger=e\). On atomic formulas, set
	\begin{align*}
		(t=u)^\dagger&=\Eq(t^\dagger,u^\dagger),&
		(t\in Y)^\dagger&=\Mem(Y,t^\dagger),\\
		S(t,u)^\dagger&=\Succ(t^\dagger,u^\dagger),&
		A(t,u,v)^\dagger&=\Add(t^\dagger,u^\dagger,v^\dagger),\\
		M(t,u,v)^\dagger&=\Mult(t^\dagger,u^\dagger,v^\dagger).
	\end{align*}
	Quantifiers are translated as follows: \((\exists x\,\varphi)^\dagger = \EEx{x}\varphi^\dagger\), \((\exists Y\,\varphi)^\dagger = \EEx{Y}\varphi^\dagger\). For Boolean connectives, $\dagger$ is defined homomorphically. Universal quantifiers are translated by $(\forall x\,\varphi)^\dagger=\AAll{x}\varphi^\dagger$ and $(\forall Y\,\varphi)^\dagger=\AAll{Y}\varphi^\dagger$.
\end{definition}

\begin{definition}
	Assume $\CGS,\varepsilon,s_0\models\Ext$. Define $0^D=d_{\varepsilon(e)}$ and the relations on $D$ by
	\begin{align*}
		S^D(d_f,d_g)&\quad\Longleftrightarrow\quad\Succ(f,g),\\
		A^D(d_f,d_g,d_h)&\quad\Longleftrightarrow\quad\Add(f,g,h),\\
		M^D(d_f,d_g,d_h)&\quad\Longleftrightarrow\quad\Mult(f,g,h),
	\end{align*}
	for $f,g,h\in\Str_{\CGS}(s_0)$. Corollary \ref{corollary-eq} and $\Ext$ ensure that these relations do not depend on the chosen strategy representatives. Finally, define \(\mathcal D = (D,0^D,S^D,A^D,M^D)\). 
\end{definition}

\begin{lemma}\label{lem-interpretation}
Assume $\CGS,s_0\models\Good$ and $\CGS,\varepsilon,s_0\models\Ext$. Let $\varphi$ be a second-order formula in the language $\mathsf{RelAr}$. Let $\eta:\{e\} \cup \operatorname{free}_1(\varphi) \cup \operatorname{free}_2(\varphi) \to \Str_{\CGS}(s_0)$ be a strategy assignment such that $\eta(e) = \varepsilon(e)$. Define $\nu_1(x) = d_{\eta(x)}$ for $x \in \operatorname{free}_1(\varphi)$ and $\nu_2(Y) = \widehat{\eta(Y)}$ for $Y \in \operatorname{free}_2(\varphi)$. Then
\(\mathcal D, \nu_1, \nu_2 \models \varphi
\Longleftrightarrow
\CGS,\eta,s_0\models\varphi^\dagger\).

\end{lemma}

\begin{proof}
The cases where $\varphi$ is atomic follow from Corollary~\ref{corollary-eq} and the definitions of the relation interpretations, using $\Ext$. Boolean cases are immediate. 

Let $\varphi=\exists x\,\psi$. If $\mathcal D,\nu_1,\nu_2\models\exists x\,\psi$, choose $d\in D$ such that $\mathcal D,\nu_1[x\mapsto d],\nu_2\models\psi$. By the definition of $D$, there is $f\in\Str_{\CGS}(s_0)$ with $d_f=d$. The assignment $\eta[x\mapsto f]$ corresponds precisely to $\nu_1[x\mapsto d]$ and $\nu_2$. By the induction hypothesis,
$\CGS,\eta[x\mapsto f],s_0\models\psi^\dagger$, hence
$\CGS,\eta,s_0\models\EEx{x}\psi^\dagger$. The converse is shown by reading these lines backwards.

Let $\varphi=\exists Y\,\psi$. If $\mathcal D,\nu_1,\nu_2\models\exists Y\,\psi$, choose $U\subseteq D$ such that $\mathcal D,\nu_1,\nu_2[Y\mapsto U]\models\psi$. Lemma~\ref{lem-powerset} gives $F\in\Str_{\CGS}(s_0)$ with $\hat F=U$. The assignment $\eta[Y\mapsto F]$ corresponds to $\nu_1$ and $\nu_2[Y\mapsto U]$. Therefore, the induction hypothesis yields
$\CGS,\eta[Y\mapsto F],s_0\models\psi^\dagger$, and thus
$\CGS,\eta,s_0\models\EEx{Y}\psi^\dagger$. Again, the converse is shown by reading the lines backwards.
\end{proof}

\begin{lemma}\label{lem-standard}
If $\mathcal D \models\mathsf{PA}_2$, then $\mathcal D$ is isomorphic to $\mathcal N$. 
\end{lemma}
\begin{proof}
By~\eqref{eq-s-total}--\eqref{eq-s-zero}, $S^D$ is the graph of an injective total function $s:D\to D$ such that $0^D$ is outside its range. The set \(
C := \{s^n(0^D) \mid n\in\NN\}\subseteq D
\)
contains zero and is closed under $s$. Then,~\eqref{eq-induction} instantiated with $U = C$ yields that $C=D$. Thus the function $\iota:\NN \to D$ defined by $\iota(n)=s^n(0^D)$ is surjective. If $\iota(m)=\iota(n)$ with $m<n$, injectivity of $s$ would give $0^D=s^{n-m}(0^D)$, contrary to~\eqref{eq-s-zero}. Hence $\iota$ is a bijection preserving zero and successor. The axioms \eqref{eq-a-total}--\eqref{eq-m-rec} guarantee that $A^D(\iota(m),\iota(n),\iota(k))$ is equivalent to $m+n=k$ (the proof is by induction on $n$) and, similarly, $M^D(\iota(m),\iota(n),\iota(k))$ is equivalent to $mn=k$.
\end{proof}

\begin{definition}
	For a closed second-order arithmetic sentence $\theta$, let
	\begin{equation}\label{eq-reduction}
		\Phi_\theta:=\Good \wedge \EEx{e}\bigl(\Ext\wedge\mathsf{PA}_2^\dagger \wedge\theta^\dagger\bigr).
	\end{equation}
\end{definition}

\begin{theorem}\label{thm:lower}
For every closed sentence $\theta$ of second-order arithmetic,
\[
\mathcal N\models\theta
\quad\Longleftrightarrow\quad
\Phi_\theta ~\text{is satisfiable}.
\]
\end{theorem}
\begin{proof}
\leavevmode

$(\Longleftarrow)$ Suppose $\CGS,s_0\models\Phi_\theta$. There is $\varepsilon : \{e\} \to \Str_{\CGS}(s_0)$ such that $\CGS,\varepsilon,s_0 \models \Ext\wedge\mathsf{PA}_2^\dagger \wedge\theta^\dagger$. Using $\CGS,\varepsilon,s_0$, define the structure $\mathcal D$ as above. Lemmas~\ref{lem-interpretation} and~\ref{lem-standard} show that it is isomorphic to $\mathcal N$. Lemma \ref{lem-interpretation} applied to $\theta^\dagger$ also gives that $\mathcal D \models \theta$, hence $\mathcal N \models \theta$.

$(\Longrightarrow)$ We define a CGS \(\CGS = \langle \AP,\Ag,\Ac,\St,\tr,\ap,s_0\rangle\) satisfying $\Phi_\theta$. Take $\Ac=\NN$ and $\St = \{s_0\} \cup \{ t_{i,j,k} \mid i,j,k\in\NN\} \cup \{ u_{\vartheta,\zeta} \mid \vartheta,\zeta\in\{0,1\}\}$. For a decision $(i,j,k,v,w)$ set
\begin{align*}
\tr(s_0,i,j,k,v,w)&=t_{i,j,k},\\
\tr(t_{i^\prime,j^\prime,k^\prime},i,j,k,v,w)&=u_{v\bmod2,w\bmod2},\\
\tr(u_{\vartheta,\zeta},i,j,k,v,w)&=u_{\vartheta,\zeta}.
\end{align*}
$\ap(t_{i,j,k})$ is the subset of $\{p_S,p_A,p_M\}$ such that
\[
p_S \in \ap(t_{i,j,k}) \iff j=i+1,\quad p_A \in \ap(t_{i,j,k}) \iff k=i+j,\quad p_M \in \ap(t_{i,j,k}) \iff k=ij.
\]
At $u_{\vartheta,\zeta}$, let $p_b$ hold exactly when $\vartheta=1$ and $p_c$ exactly when $\zeta=1$; no arithmetic proposition holds there. No propositions hold at $s_0$. This completes the definition of $\CGS$. Let us check that $\CGS,s_0 \models \Phi_\theta$.

\begin{itemize}
	\item The states reachable from $s_0$ in one move are exactly $t_{i,j,k}$. At such a state, the transition
	\[
	\tr(t_{i,j,k},i^\prime,j^\prime,k^\prime,v,w)=u_{v\bmod2,w\bmod2}
	\]
	depends only on $v,w$. At the resulting state, $p_b$ holds iff $v$ is odd, and $p_c$ holds iff $w$ is odd. Each agent can choose either parity independently. Thus $\CGS,t_{i,j,k}\models\Test$ by Lemma~\ref{lem-local}, and $\CGS,s_0\models L$.
	
	\item Let $(f_1,f_2,f_3,f_b,f_c)$ be an $s_0$-total strategy assignment and put $t=t_{f_1(s_0),f_2(s_0),f_3(s_0)}$. Then
	\begin{align*}
	&\CGS,(f_1,f_2,f_3,f_b,f_c),s_0\models\Next\Next p_c\\
	&\quad\Longleftrightarrow\quad
	\CGS,(f_1,f_2,f_3,f_b,f_c)_{\rightarrow s_0t},t\models\Next p_c\\
	&\quad\Longleftrightarrow\quad
	\CGS,u_{f_b(s_0t)\bmod2,f_c(s_0t)\bmod2}\models p_c\\
	&\quad\Longleftrightarrow\quad f_c(s_0t)\bmod2=1.
	\end{align*}
	Consequently, the truth of \(\Next \Next p_c\) at $s_0$ does not depend on $b$'s strategy, so $J_b$ is true at $s_0$. Symmetrically, we can show that $\CGS,s_0 \models J_c$, hence $\CGS,s_0 \models \Good$.
	
	\item We need to show that there is a strategy assignment $\varepsilon:\{e\} \to \Str_{\CGS}(s_0)$ such that $\CGS, \varepsilon, s_0 \models \Ext\wedge\mathsf{PA}_2^\dagger \wedge\theta^\dagger$. Let $\varepsilon(e)$ be a constant-zero strategy (it outputs action $0$ on each input). In this case, $d_f=t_{f(s_0),0,0}$ for every $f \in \Str_{\CGS}(s_0)$. Therefore, $\Eq(f,g)$ holds iff $f(s_0) = g(s_0)$, hence $\Eq$ identifies strategies with their actions on the initial state. $\Succ(f,g)$ holds iff $\CGS,(f,g,\varepsilon(e),\varepsilon(e),\varepsilon(e)), s_0 \models \Next p_S$ iff $p_S \in \ap(t_{f(s_0),g(s_0),0})$ iff $g(s_0) = f(s_0)+1$. Similarly, $\Add(f,g,h)$ holds iff $h(s_0) = f(s_0) + g(s_0)$ and $\Mult(f,g,h)$ holds iff $h(s_0) = f(s_0) g(s_0)$. This immediately yields that $\CGS, \varepsilon, s_0 \models \Ext\wedge\mathsf{PA}_2^\dagger$. 
	
	Define $\mathcal D$ from $\CGS,\varepsilon,s_0$ as before. Then, $\mathcal D$ is isomorphic to $\mathcal N$. By Lemma~\ref{lem-interpretation}, since $\mathcal N \models \theta$ and therefore $\mathcal D \models \theta$, $\CGS,\varepsilon,s_0 \models \theta^\dagger$. Thus, $\Phi_\theta$ is satisfiable in $\CGS,s_0$. \qedhere
\end{itemize}
\end{proof}

\paragraph{Analyzing the syntax of the formulae.}

It is straightforward to verify that $I_i,P_i,\Test,L,J_i$ are all \(\SLBGX\) formulae, hence so is \(\Good\). Formally speaking, one should delete the vacuous quantifier $\AAll{y_i}$ in $I_i$ because $y_i$ does not occur in the following part of the formula (it is replaced by $x_i$). 

To handle the formula \(\EEx{e}\bigl(\Ext\wedge\mathsf{PA}_2^\dagger \wedge\theta^\dagger\bigr)\), we rename bound strategy variables so that distinct quantifiers bind distinct variables and then convert it into the prenex normal form, treating its subformulae of the form \(\B(\bar x) \xi \) as atomic formulae. The result is a formula \(\EEx{e} Q_1 \ldots Q_k \Xi\) where $Q_1, \ldots, Q_k$ are quantifiers and $\Xi$ is a Boolean combination of formulae of the form \(\B(\bar x) \xi \). This is a \(\SLBGX\) formula. Let us denote \(\Good \wedge \EEx{e} Q_1 \ldots Q_k \Xi\) by \(\Phi^\prime_\theta\). Formally, the mapping \(\theta \mapsto \Phi^\prime_\theta\) is non-injective because different \(\theta\)'s can have the same prenex form. However, this can easily be fixed by adding some vacuous true statement to \(\Phi^\prime_\theta\). E.g.~take some valid proposition $\top = (p_c \leftrightarrow p_c)$ and define the new mapping \(\theta \mapsto \Phi^\prime_\theta \wedge \underbrace{\top \wedge \ldots \wedge \top}_{\ulcorner \theta \urcorner ~ \text{times}}\) where \(\ulcorner \theta \urcorner\) is the G\"odel number of \(\theta\) under some coding.

Theorem \ref{theorem:main} follows from Theorem \ref{thm:injective-lower} and Proposition \ref{prop:upper}, by Myhill's isomorphism theorem. Therefore, \(\SL\) satisfiability is not in \(\Sigma^1_1\), which entails that \(\SL\) cannot be axiomatized even using effective \(\omega\)-rules. By the latter we mean that, given the conclusion of the rule and $n \in \NN$, one can compute the \(n\)-th premise of the rule. Indeed, if such an axiomatization existed, the provability problem would belong to \(\Pi^1_1\), hence satisfiability would belong to \(\Sigma^1_1\), which is not the case.

Theorem \ref{theorem:main} resolves negatively the following conjecture from \cite{CattaGalimullinMurano2025}:
\begin{quote}
	`The general sentiment is that $\SL$ is still not recursively axiomatisable, but to show this result, one will have to employ richer features of $\SL$, \textbf{beyond its next-time fragment}.'
\end{quote}
Our construction shows that the next-time fragment of even \(\SLBGX\) results in the lack of recursive axiomatizability.

\section*{Disclosure on Applying AI Tools} 

The proof of the main result was generated by GPT 5.6-Sol and GPT 6-Astra. I have checked and understood all details of the proof, polished the AI-generated text and assume full responsibility for its accuracy. 

Retrospectively, the proof of the main result does not look sophisticated, and it being AI-generated raises the question of whether I should disseminate it at all. My answer is yes---at least as a preprint. First, the complexity of $\SL$ satisfiability has been of interest to the community, and there has been some discussion around it. The right answer to this problem was not obvious for me a priori, especially for the case of \(\SLBG\). Settling the problem and seeing the right way to approach it might be useful for further studies of logics for strategic reasoning.

Secondly, and more importantly, this paper is just an initial part of the ongoing work. A broader goal I want to contribute to is understanding the landscape of logics like \(\mathsf{CTL}\), \(\mathsf{ATL}\) and many others in terms of their complexity and delineating a border between feasible and highly complex logics. Concerning the satisfiability problem, it is interesting where the border between decidable and highly undecidable fragments lies and whether there are fragments of intermediate complexity. I hope to continue working on this topic and to make use of the particular result presented in the paper to provide other, more substantial insights.

I believe that I contributed to this work by (a) \emph{wanting to answer this question}, (b) choosing what I believed to be the right direction to explore (I asked AI in the first place to provide a reduction from second-order arithmetic to \(\SL\) satisfiability) and, of course, by (c) checking and understanding it, and making it readable. Whether this is enough to attribute the result to me is for the reader to decide.

\bibliographystyle{apalike}
\bibliography{sl_complexity_references}

\end{document}